\documentclass[12pt]{amsart}
\usepackage[utf8]{inputenc}
\usepackage{amssymb,amsmath,amsthm, bbm, mathrsfs}
\usepackage[T1]{fontenc}
\usepackage{graphicx}
\usepackage{color}
\usepackage{t-angles}
\usepackage[framemethod=tikz]{mdframed}
\definecolor{refkey}{rgb}{0,0,1}
\definecolor{labelkey}{rgb}{1,0,0}
\usepackage[top=2cm,left=3cm,right=2cm,bottom=2cm, nohead]{geometry}
\newtheorem{thm}{Theorem}[section]

\newtheorem{lem}[thm]{Lemma}
\newtheorem{cor}[thm]{Corollary}
\theoremstyle{definition}
\newtheorem{defn}[thm]{Definition}
\newtheorem{rem}[thm]{Remark}
\numberwithin{equation}{section}
\newcommand{\A}{\mathcal{A}}        
\renewcommand{\a}{\chi}             

\newcommand{\C}{\mathbb{C}}         
\newcommand{\R}{\mathbb{R}}         
\newcommand{\Z}{\mathbb{Z}}         
\newcommand{\cop}{\Delta}           
\newcommand{\ox}{\otimes}           
\DeclareMathOperator{\id}{id}       
\begin{document} 
\title{Braided Hopf algebras from twisting.} 
\author[A.\ Bochniak]{Arkadiusz Bochniak}
\address{Institute of Physics, Jagiellonian University,
prof.\ Stanis\l awa \L ojasiewicza 11, 30-348 Krak\'ow, Poland \newline\indent
Copernicus Center for Interdisciplinary Studies, Szczepańska 1/5, 31-011 Kraków, Poland}

\author[A.\ Sitarz]{Andrzej Sitarz${}^*$} 
\address{Institute of Physics, Jagiellonian University,
prof.\ Stanis\l awa \L ojasiewicza 11, 30-348 Krak\'ow, Poland.\newline\indent
 Institute of Mathematics of the Polish Academy of Sciences,
\'Sniadeckich 8, 00-656 Warszawa, Poland.}
\thanks{${}^*$Partially supported by NCN grant 2015/19/B/ST1/03098}
\email{andrzej.sitarz@uj.edu.pl}   
\begin{abstract} 
We show that a class of braided Hopf algebras, which includes the braided $SU_q(2)$ of 
\cite{KMRW14}, is obtained by twisting. We show further examples and demonstrate that 
twisting of bicovariant differential calculi gives braided bicovariant differential 
calculi.
\end{abstract} 
\maketitle 
\section{Introduction}

Hopf algebras are a natural generalization of symmetries in mathematics, providing an unifying
picture of discrete groups, Lie groups and Lie algebras. Apart from their role as symmetries of
hypothetical quantum space-time in models of quantum gravity, they appear in the description
of renormalization and arise in integrable models and around solutions of Yang-Baxter equation.
Some interesting generalizations of Hopf algebras are, however, slightly different, as their
coproduct is an algebra homomorphism only if one passes to the braided category. Such objects are
natural candidates to generalize symmetries of quantum physical systems with nontrivial 
statistics, like anyons, but they appear frequently as models of noncommutative spaces and
noncommutative geometries \cite{CoLa}. So far the existence of braided symmetries there has 
been not studied there.
 
In \cite{KMRW14} the authors introduced and studied a family of deformations of the algebra of 
functions on the $SU(2)$ group, which are braided Hopf algebras. Though the deformation and the
description in the above paper was in the $C^\ast$-algebraic language it is interesting to notice 
that the structure is easily translated for the Hopf algebra of the quantum $SU(2)$ group. 

Our original motivation was to construct braided differential calculi for the braided $SU_q(2)$. 
In this note we show that the general framework of twisting provides the answer, that twisting
applied to bicovariant differential calculi over the Hopf algebra (as classified by Woronowicz in 
\cite{Wo89}) gives braided bicovariant calculus.

We provide also a further class of examples that include braided $SU_q(n)$ and braided quantum 
double torus.

Throughout the paper we work with braided Hopf algebras, which are Hopf algebras in a braided tensor 
category, with a given braiding $\Psi$ as defined in \cite{majid95}.
\section{Twisting Hopf Algebras}

Let $H$ be a Hopf algebra with the coproduct $\Delta$, the 
counit $\epsilon$ and the antipode $S$. Let $\a: H \to \C[z,z^{-1}]$ 
be a Hopf algebra homomorphism, where $\C[z,z^{-1}]$ is the
Hopf algebra of the group algebra of $\Z$. 

\begin{lem}
The following maps define left and right coactions of $\C[z,z^{-1}]$ on $H$, 
$$ \cop_L = (\a \ox \id) \cop, \quad \quad  \cop_R = (\id \ox \a) \cop, $$
which are equivariant under the coproduct in $H$.
\end{lem}
\begin{proof}
We check only $\Delta_L$. Using the fact that $\a$ is Hopf algebra morphism:
$$ 
\begin{aligned}
(\id \otimes \Delta_L) \Delta_L(x) & = (\id \ox \a \ox \id) (\id \ox \Delta) (\a \ox \id) \Delta(x)  
= \a(x_{(1)}) \ox \a(x_{(2)}) \ox x_{(3)} \\
& = \left(\Delta_{\C[z,z^{-1}]}\ox \id\right) (\a(x_{(1)})  \ox x_{(2)} ) =
 ( \Delta_{\C[z,z^{-1}]} \ox \id ) \Delta_L(x),
\end{aligned}
$$
and
$$ (\epsilon_{\C[z,z^{-1}]} \ox \id) \cop_L = (\epsilon_{\C[z,z^{-1}]} \ox \id) (\a \ox \id) \cop
= (\epsilon \ox \id) \cop = \id. $$
The equivariance under the coproduct in $H$ means:
$$ 
\begin{aligned}
(\id \otimes \Delta) \Delta_L(x) & = (\id \otimes \Delta) (\a(x_{(1)}) \ox x_{(2)}) =
\a(x_{(1)}) \ox x_{(2)} \ox x_{(3)} \\
& = \Delta_L(x_{(1)}) \ox x_{(2)}   = ( \Delta_L \ox \id ) \Delta (x).
\end{aligned}
$$
\end{proof}
\begin{defn}
We say that $x \in H$ is a homogeneous element of degrees 
$\mu(x),\nu(x)$, where $\mu(x), \nu(x) \in \Z$ if:
$$ \cop_L(x) = z^{\mu(x)} \ox x, \quad \quad \cop_R(x) = x \ox z^{\nu(x)}. $$
We also define the index 
$$\delta(x)=\mu(x)-\nu(x).$$
\end{defn}

We shall study now properties of the maps $\mu$ and $\nu$. 

\begin{lem}\label{add}
The degrees are additive, that is:
$$ \mu(xy) = \mu(x)+\mu(y), \qquad \qquad \nu(xy) = \nu(x)+\nu(y). $$
\end{lem}
This follows directly from the definition of degrees and the fact that 
$\C[z,z^{-1}]$ is a $\Z$-graded algebra.


{}From now on we assume that the Hopf algebra $H$ has a countable basis of 
homogeneous elements, so that any element of $H$ could be presented as a finite
sum of elements with fixed degrees $\mu$ and $\nu$.
\begin{lem}
\label{degprop}
We have the following identities for any $x \in H$ and  
$\Delta x = \sum_i x_{(1)}^i \otimes  x_{(2)}^i$, where $x_{(1)}^i$
and $x_{(2)}^i$ are homogeneous:
$$ \mu(x) = \mu(x_{(1)}^i), \qquad \nu(x) =  \nu(x_{(2)}^i), 
\qquad \hbox{for any $i$}, $$
and
$$ \mu(x_{(2)}^i) = \nu(x_{(1)}^i), \qquad \hbox{for any $i$}.$$
\end{lem}
\begin{proof}
From the identity:
$$  (\id \ox \Delta) \Delta_L = ( \Delta_L \ox \id) \Delta, $$ 
we see that for a homogeneous $x$ of degree $\mu(x)$:
$$ (\hbox{id} \otimes \Delta) \Delta_L(x) =  z^{\mu(x)} \otimes \Delta(x) = 
\sum_i z^{\mu(x)} \otimes x_{(1)}^i \otimes  x_{(2)}^i, $$
but on the other hand:
$$  ( \Delta_L \ox \id) \Delta (x) = \sum_i \Delta_L(x_{(1)}^i) \otimes x_{(2)}^i $$
By comparing these two expressions and taking into account that each $x_{(1)}^i$ in the sum 
are homogeneous and are the basis of the algebra, we obtain 
$$ \mu(x) = \mu(x_{(1)}^i), \qquad \hbox{for every} \, i. $$
Similarly we prove the remaining identities.
\end{proof}
In particular, as a consequence we have:
\begin{cor}
If $x = \sum_i x_{(1)}^i \otimes x_{(2)}^i$ then, for any $i$: 
\begin{equation}
\delta(x) = \delta(x_{(1)}^i) +  \delta(x_{(2)}^i). 
\label{degad}
\end{equation} 
\end{cor}
\begin{lem}
\label{antipode}
For each homogeneous $x$ we have:
\begin{equation}
\mu(S(x))=-\nu(x), \qquad \nu(S(x))=-\mu(x).
\end{equation} 
\end{lem}
\begin{proof}
Since for each element $x$ we have:
$$ \Delta (S(x)) = S(x_{(2)}) \ox S(x_{(1)}), $$
composing it with $\a$, which is a Hopf algebra morphism, so
$\a(S(x)) =S_{\C[z,z^{-1}]} (\a(x))$ and using the form of the antipode 
on ${\C[z,z^{-1}]}$ algebra we obtain that for a homogeneous $x$
with degrees $\mu(x),\nu(x)$ we have:
$$\Delta_L(S(x)) = S_{\C[z,z^{-1}]}(z^{\nu(x)})  \ox S(x) = z^{-\nu(x)} \ox S(x),$$
$$\Delta_R(S(x)) = S(x) \ox S_{\C[z,z^{-1}]}(z^{\mu(x)}) = S(x) \ox z^{-\mu(x)}.$$
\end{proof}
One of the consequences of the compatibility of coactions with the coproduct and the antipode is the following fact:
\begin{lem}
\label{counit}
For any homogeneous element $x$,  $\epsilon(x)=0$ if $\delta(x)\not=0$.
\end{lem}
\begin{proof}
Using the property that $\Delta, \epsilon, \a$ are morphisms we have:
$$ m \circ (\a \ox \epsilon) \Delta(x) = \a(x_{(1)}) \epsilon(x_{(2)}) = \a(x). $$
However, for a homogeneous $x$:
$$ m (\a \ox \epsilon) \Delta(x) = z^{\mu(x)} \epsilon(x). $$
On the other hand:
$$ m (\epsilon \ox \a) \Delta(x) = \epsilon(x_{(1)}) \a(x_{(2)})  = \a(x), $$
and
$$ m (\epsilon \ox \a) \Delta(x) = \epsilon(x)z^{\nu(x)}. $$
Therefore, for any homogeneous $x$ we have:
$$ z^{\delta(x)} \epsilon(x) = \epsilon(x). $$
\end{proof}
\subsection{The twisted product}

\begin{lem}\label{cocycle}
The following map defined on homogeneous elements $x,y \in H$ for any $\phi \in \R$:
$$\Phi (x, y) = e^{i \phi(\mu(x) \nu(y)  - \nu(x) \mu(y))}$$
extends to a two-cocycle on $H$, that is a map $ \Phi: H \otimes H \to U(1)$, 
which satifies for all $x,y,z \in H$ :
$$ \Phi(x, yz) \Phi(y,z) = \Phi(x,y) \Phi(xy,z).$$
\end{lem}
\begin{proof}
Since for homogeneous elements $x,y \in \A$ we have Lemma 
\ref{add}, we further compute,
$$
\begin{aligned}
\Phi(x, yz) \Phi(y,z) &=  e^{i \phi(\mu(x) \nu(yz)  - \nu(x) \mu(yz))}  
e^{i \phi(\mu(y) \nu(z)  - \nu(z) \mu(y))} \\
 &=  e^{i \phi(\mu(x) (\nu(y)+\nu(z))  - \nu(x) (\mu(y)+\mu(z)))} 
 e^{i \phi(\mu(y) \nu(z)  - \nu(z) \mu(y))}\\
& =  e^{i \phi(\mu(x) \nu(y) - \nu(x) \mu(y))} 
e^{i \phi( \mu(xy) \nu(z))  - \mu(z) \nu(xy))} \\
 & =  \Phi(x,y) \Phi(xy,z).
\end{aligned}
$$
\end{proof}
As a natural consequence we have,
\begin{lem}
\label{algebra}
The following product defined on homogeneous elements of 
$H$ extends linearly to an associative product on $H$:
$$ x \ast y =   \Phi(x, y) \, x \cdot y. $$
\end{lem}

Moreover, note that since $\mu(1)=\nu(1)=0$ the algebra we obtain is unital (with the same unit 
$1$ as in $H$). We will denote this algebra by $H_\phi$.
\subsection{The braiding}
We introduce the following prebraiding on $H \otimes H$.
\begin{lem}
The map defined on homogeneous elements $x,y \in H_\phi$:
\begin{equation}
\Psi (x \otimes y) = e^{2 i\phi \delta(x)\delta(y)} y \otimes x, 
\label{braid}
\end{equation}
defines a prebraiding on $H\ox H$.
\end{lem}
\begin{proof}
We can introduce degrees of tensor products of homogeneous elements, in the
following way. Taking  $x= x_1 \otimes x_2 \otimes \cdots \otimes x_n$,
we define:
$$ \delta \left( x \right) = \sum_{i=1}^n \delta(x_i). $$
then, for $x$ as above and $y = y_1  \otimes \cdots \otimes y_k$
we define the braiding on tensors as:
$$
\Psi \left( x \otimes y  \right)
= e^{ 2 i \phi \delta(x) \delta(y)}   
\left( y \otimes x  \right). 
$$
In this way we obtain a prebraided monoidal structure. Indeed one can easily check that
$$ \Psi\left( (x \otimes y) \otimes z \right) = (\Psi \otimes \id) (\id \otimes \Psi)(x \otimes y \otimes z), $$
$$ \Psi\left( x \otimes (y \otimes z) \right) = (\id \otimes \Psi) (\Psi \otimes \id)(x \otimes y \otimes z), $$
if $\delta$ satisfies $\delta(x\ox y)=\delta(x)+\delta(y)$.
\end{proof}

From this prebraided structure defined above we obtain a braided monoidal category of vector spaces, if as morphisms we take maps which commute with $\Psi$. For $F:H\rightarrow H$ it is sufficient to check that $\delta \circ F = \delta$. This follows from the fact that in this case the naturality conditions:
$$ \Psi (F \otimes \id) = (\id \otimes F) \Psi, \qquad
\Psi  (\id \otimes F) = (F \otimes \id) \Psi, $$ are satisfied.
We shall denote this category by 
$\mathcal{H}_{\phi}^{\ox}\equiv \left(\mathcal{H}_{\phi}^{\ox}, %
\mathfrak{Mor}(\mathcal{H}_{\phi}^{\ox}),\Psi\right)$.

As a result we have that $\Psi$ is a Yang-Baxter operator, which can be also checked directly:
$$ 
\begin{aligned}
(\id \otimes \Psi) (\Psi \otimes \id) &(\id \otimes \Psi) (x \otimes y \otimes z)= \\
& = e^{2i\phi (\delta(x) \delta(y) + \delta(x) \delta(z) + \delta(y) \delta(z))} (z \otimes y \otimes x)  \\
& = (\Psi \otimes \id) (\id \otimes \Psi)  (\Psi \otimes \id) (x \otimes y \otimes z),
\end{aligned}
$$
This follows again from the properties of $\delta$.

\subsection{The coproduct}
\begin{thm} 
\label{coass}
Using a similar the one used to define the braiding we construct the coproduct on $H_\phi$:
\begin{equation}
\Delta_\phi(x) = \sum_j e^{ i \phi \delta(x_{(1)}^j)  \delta(x_{(2)}^j)}  x_{(1)}^j \otimes  x_{(2)}^j.
\label{coprod}
\end{equation}
\end{thm}
\begin{proof}
For simplicity we skip here the sum and the summation indices, one should remember that
we are using the basis of homogeneous elements and the coproduct $\Delta(x) = 	x_{(1)} \otimes  x_{(2)}$
is in fact a well-defined unique expression and it is meant to be $\sum_j x_{(1)}^j \otimes  x_{(2)}^j$ where
all terms are homogeneous elements of the basis.

Let us verify that the coproduct (\ref{coprod}) is coassociative:
$$ (\hbox{id} \otimes \Delta_\phi) \Delta_\phi(x) = e^{ i \phi \delta(x_{(1)})  \delta(x_{(2)})}  
 e^{ i \phi \delta(x_{(2,1)})  \delta(x_{(2,2)})}
x_{(1)} \otimes  x_{(2,1)} \otimes x_{(2,2)}, $$
but 
$$ (\Delta_\phi  \otimes \hbox{id}) \Delta_\phi(x) = e^{ i \phi \delta(x_{(1)})  \delta(x_{(2)})}  
 e^{ i \phi \delta(x_{(1,1)})  \delta(x_{(1,2)})}
x_{(1,1)} \otimes  x_{(1,2)} \otimes x_{(2)}, $$
Then, both sides are equal to each other if Lemma \ref{degprop} is satisfied.
\end{proof}
\begin{thm}
\label{morph}
The coproduct $\Delta_\phi$ (\ref{coprod}) satisfies
$$\cop_{\phi}(x \ast y)=\cop_{\phi}(x)\ast\cop_{\phi}(y),$$
where $(a\ox b)\ast (c \ox d):=a \Psi(b \ox c) d$.
\end{thm}
\begin{proof}
Let us take two homogeneous elements $x,y \in H$ with their coproducts:
$$ \Delta x = x_{(1)} \otimes  x_{(2)}, \qquad \qquad  \Delta y = y_{(1)} \otimes  y_{(2)}. $$
We compute first:
$$
\begin{aligned}
\Delta_\phi (x \ast y) & =  e^{i \phi(\mu(x) \nu(y)  - \nu(x) \mu(y))} \Delta_\phi(xy) \\
& = e^{i \phi(\mu(x) \nu(y)  - \nu(x) \mu(y))} 
e^{ i \phi (\delta(x_{(1)}) +  \delta(y_{(1)}))(\delta(x_{(2)}) + \delta(y_{(2)})}
 \left( x_{(1)} y_{(1)} \otimes  x_{(2)} y_{(2)} \right).
\end{aligned}
$$
On the other hand:
$$
\begin{aligned}
\Delta_\phi(x) \ast \Delta_\phi(y) &= 
e^{ i \phi \delta(x_{(1)})  \delta(x_{(2)})}  
e^{ i \phi \delta(y_{(1)})  \delta(y_{(2)})}  
\left( x_{(1)} \otimes  x_{(2)} \right) \ast \left( y_{(1)} \otimes  y_{(2)} \right) \\
&= 
e^{ i \phi \delta(x_{(1)})  \delta(x_{(2)})}  
e^{ i \phi \delta(y_{(1)})  \delta(y_{(2)})}  
e^{2 i \phi \delta(x_{(2)})  \delta(y_{(1)})}  \cdot \\
& \,\, \cdot e^{ i \phi (\mu(x_{(1)}) \nu(y_{(1)})  - \mu(y_{(1)}) \nu(x_{(1)})}   
e^{ i \phi (\mu(x_{(2)}) \nu(y_{(2)})  - \mu(y_{(2)}) \nu(x_{(2)})}   
\left(  x_{(1)} y_{(1)} \otimes  x_{(2)} y_{(2)} \right).
\end{aligned}
$$
A simple computation, using Lemma \ref{degprop} allows to demonstrate that both scalar factors are identical, hence the map $\Delta_\phi$ satisfies $\cop_{\phi}(x \ast y)=\cop_{\phi}(x)\ast\cop_{\phi}(y)$.
\end{proof}
\begin{thm}
\label{braidedHopf}
$H_\phi$ is a braided bialgebra in braided category with the braiding given by (\ref{braid}).
\end{thm}
\begin{proof}
Let us denote $m_\ast (x\otimes y)=x \ast y$. From Lemma \ref{algebra} we know that $(H_{\phi},m_\ast,1)$ is an algebra. Theorem \ref{coass} implies coassociativity of $\Delta_\phi$, properties of counit follows from Lemma \ref{counit}. It means that $(H_\phi,\Delta_\phi,\epsilon)$ forms coalgebra.

Now, we check that the structures 
$m_\ast,\Delta_{\phi},\varepsilon$ commute with $\Psi$, i.e.
\begin{align}
\Psi(\mathrm{id}\otimes m_\ast)=(m_\ast\otimes \mathrm{id})(\mathrm{id}\otimes \Psi)
(\Psi\otimes \mathrm{id}) \label{i1} \\
\Psi(m_\ast\otimes\mathrm{id})=(\mathrm{id}\otimes m_\ast)(\Psi\otimes\mathrm{id})
(\mathrm{id}\otimes\Psi) \label{i2}\\
(\mathrm{id}\otimes \Delta_{\phi})\Psi=(\Psi\otimes\mathrm{id})(\mathrm{id}\otimes \Psi)(\Delta_{\phi}\otimes \mathrm{id}) \label{i3}\\
(\Delta_{\phi}\otimes\mathrm{id})\Psi=(\mathrm{id}\otimes \Psi)(\Psi\otimes\mathrm{id})(\mathrm{id}\otimes\Delta_{\phi})\label{i4}
\end{align}
Indeed, one can easily compute that
$$\Psi(\mathrm{id}\otimes m_\ast)(x\otimes y\otimes z)=e^{i\phi\left(\mu(y)\nu(z)-\mu(z)\nu(y)\right)}e^{2 i\phi\delta(x)\delta(yz)}yz\otimes x$$ 
and
$$(m_\ast\otimes \mathrm{id})(\mathrm{id}\otimes \Psi)(\Psi\otimes \mathrm{id})=e^{i\phi\left(\mu(y)\nu(z)-\mu(z)\nu(y)\right)}e^{2 i\phi\delta(x)\delta(z)}e^{2 i\phi\delta(x)\delta(y)}yz\otimes x.$$
Hence we obtain (\ref{i1}) and similarly (\ref{i2}). To prove (\ref{i3}),(\ref{i4}) we do analogous computation but in the 
last step we use (\ref{degad}).

Moreover, by Theorem \ref{morph} we know that $\Delta_\phi$ is a morphism of algebras. 
To finish the proof we need to show that $\epsilon$ is an algebra map:
$$ \epsilon( x \ast y) = \epsilon(x) \epsilon(y).$$
By definition we have:
$$ \epsilon( x \ast y) = e^{i \phi(\mu(x) \nu(y)  - \nu(x) \mu(y))} \epsilon(x) \epsilon(y). $$
The only nontrivial case is when $\epsilon(x) \not=0 $ and $\epsilon(y) \not=0$. But by Lemma \ref{counit} then $\mu(x)=\nu(x)$ and $\mu(x)=\nu(y)$ hence the above equality holds.
\end{proof}
\subsection{The antipode}
Finally, we extend an antipode $S$ defining $S_\phi:H_\phi\rightarrow H_{\phi}$ in the following way.

\begin{thm}
\label{antithm}
The antipode, defined on homogeneous elements as:
$$S_{\phi}(x)=e^{i\phi\delta(x)^2}S(x),$$
makes $H_\phi$ is a braided Hopf algebra in $\mathcal{H}_{\phi}^{\otimes}$. 
\end{thm}

\begin{proof}
To show that it is also a Hopf algebra we need to check that $S_\phi\in \mathfrak{Mor}\left(\mathcal{H}_{\phi}^{\ox}\right)$. It is enough to check that $\delta\circ S=\delta$. But it simply follows from lemma \ref{antipode}.
\begin{equation}
(\delta \circ S)(x)=\mu(S(x))-\nu(S(x))=-\nu(x)-(-\mu(x))=\mu(x)-\nu(x)=\delta(x),
\end{equation}
i.e. $\delta\circ S=\delta$. Hence we have $S_\phi\in \mathfrak{Mor}\left(\mathcal{H}_{\phi}^{\ox}\right)$.

Next, we verify:
$$  
\begin{aligned}
m_\ast (S_\phi \otimes \hbox{id}) \Delta_\phi(x) 
&=  e^{ i \phi \delta(x_{(1)})  \delta(x_{(2)})}   S_\phi(x_{(1)}) \ast x_{(2)} \\
&=  e^{ i \phi \delta(x_{(1)})  \delta(x_{(2)})} e^{i\phi \delta(x_{(1)})^2 }  S(x_{(1)}) \ast x_{(2)} \\
&=  e^{ i \phi \delta(x_{(1)})  \delta(x_{(2)})} e^{i\phi \delta(x_{(1)})^2}  
    e^{ i \phi (-\nu(x_{(1)})  \nu(x_{(2)}) + \mu(x_{(2)}) \mu(x_{(1)}) )} S(x_{(1)}) x_{(2)} \\
&=  e^{ i \phi \mu(x_{(1)})  \delta(x) } \epsilon(x) = \epsilon(x).
\end{aligned}
$$
The last step follows from the fact that $\delta(x)$ does not vanish only on kernel of $\epsilon$, 
therefore once $\epsilon(x) \not= 0$ then the scalar factor in front $\epsilon(x)$ is $1$.
\end{proof}

We finish the section by observing that if $H$ is Hopf-$\ast$ algebra and $\a$ is a $\ast$-morphism then $H_\phi$ is a braided Hopf $\ast$-algebra with 
$$\begin{aligned}(x\ox y)^\ast=\psi(x,y)x^\ast\ox y^\ast\end{aligned},$$
where $\psi(x,y)$ is the phase factor in the braiding, i.e. $\Psi(x\ox y)=\psi(x,y)y\ox x.$
\begin{rem}
Observe, that due to the theorem of Schauenburg: \cite{Scha}, for a Hopf algebra 
in the prebraided monoidal category $(\mathcal{C},\Psi)$, we have:
$$	
\Psi = \left(m_* \otimes m_* \right)\left(S_\phi \otimes \left(\Delta_\phi \circ m_* \right)\otimes S_\phi \right)\left(\Delta_\phi \otimes \Delta_\phi \right).
$$	
We verify, that for two homogeneous elements $x,y \in H$:
$$
\begin{aligned}
\Psi(x,y) &= S_\phi(x_{(1)}) *  \left( x_{(2)} * y_{(1)} \right)_{(1)} \otimes  \left( x_{(2)} * y_{(1)} \right)_{(2)} * S_\phi (y_{(2)}) \\
 &= e^{ i\phi \zeta(x,y)} y \otimes x,
\end{aligned}
$$
where the phase $\zeta(x,y)$ reads, using shortcut $z =   x_{(2)} y_{(1)}$,
$$
\begin{aligned}
\zeta(x,y) =& \;
\delta(x_{(1)}) \delta(x_{(2)}) + \delta(y_{(1)}) \delta(y_{(2)}) + \delta(x_{(1)})^2   
+ \delta(y_{(2)})^2 \\
&+  \left( \mu(x_{(2)}) \nu(y_{(1)}) - \nu(x_{(2)}) \mu(y_{(1)}) \right)   
+ \delta(z_{(1)}) \delta(z_{(2)}) \\
& + \left( \mu(S x_{(1)}) \nu(z_{(1)}) - \mu(z_{(1)}) \nu(S x_{(1)}) \right)
     + \left( \mu(z_{(2)}) \nu(Sy_{(2)}) - \mu(Sy_{(2)}) \nu(z_{(2)}) \right)
\end{aligned}
$$
\end{rem}
Of course, the phase factor enters only when the relevant terms are non-zero, which happens only if the counit does not vanish, and that is exactly when:
$$\begin{aligned} \epsilon(x_{(1)})\epsilon((x_{(2)})_{(1)})\neq 0, \qquad \epsilon((y_{(1)})_{(2)})\epsilon(y_{(2)})\neq 0 \end{aligned}$$ 
Let's take elements with $x$ only:
$$\begin{aligned}\delta(x_{(1)})\delta(x_{(2)})+\delta( x_{(1)})^2+\delta((x_{(2)})_{(1)}) \delta((x_{(2)})_{(2)}) +(-\nu(x_{(1)})\nu((x_{(2)})_{(1)}) + \mu(x_{(1)})\mu((x_{(2)})_{(1)})) \end{aligned}$$ 
\subsection{Twisting and untwisting}

Let us note that the twisting of the product, coproduct and braiding depends 
solely on the bigrading $\mu,\nu$ that satisfied the properties from Lemmas 
\ref{add},\ref{degprop}, \ref{antipode} and Lemma \ref{counit}.

It is easy to see that neither of this properties is changed in any way if we pass to the 
algebra $H_\phi$ with the braiding (\ref{braid}). Both $H$ and $H_\phi$ share the same basis of elements
with fixed grading $\mu,\nu$ and since all products and coproducts are deformed in a 
homogeneous way, neither of above condition is changed. Therefore, $H_\phi$ can 
be further twisted using the above procedure and after a twist by an angle $\eta$
we obtain, in fact a Hopf algebra, which is a twist of the original Hopf algebra $H$ by 
$\phi+\eta$.
\section{Examples}

\subsection{The algebra $\C[z,z^{-1}]$.}\ \newline
Consider the $H=\C[z,z^{-1}]$ with $\a$ being the identity map. Since the algebra is cocommutative
then $\mu(x) \equiv \nu(x)$ for every homogeneous $x$ and therefore $H_\phi=H$ in this case.

\subsection{The algebra $C_\lambda(a,b)$.}\ \newline
Consider a unital Hopf algebra generated by $a,a^{-1},b$ with relation:
$$ a\, b = \lambda b \, a, $$
and coproduct,
$$ \Delta(a) =a \otimes a, \qquad \Delta(b) = a \otimes b + b \otimes 1, $$
counit and antipode:
$$ \epsilon(a)=1, \quad \epsilon(b)=0, \quad S(a)=a, \quad S(b) = -a^{-1} b. $$
If, additionally $\lambda$ is real then it is Hopf $\ast$-algebra with $b=b^*$ and
$a^* = a^{-1}$.

We take the map $\a: C_\lambda(a,b) \to \C[z,z^{-1}]$ as:
 $$ \a(a) =z, \qquad \a(b)=0, $$
which is a morphism of Hopf algebras (and of Hopf $\ast$-algebras if we
consider the $\ast$ structure). we easily see that:
$$ \mu(a) = \nu(a) = \mu(b) = 1, \qquad \nu(b) = 0, $$
so that $\delta(a) =0$ and $\delta(b) = 1$. 

The twisted braided Hopf algebra $C_\lambda(a,b)_\phi$ is then generated
by $a,b$ which obey the relation:
$$ a \ast b = e^{- 2 i \phi } \lambda b \ast a, $$
with the same coproduct $\Delta_\phi = \Delta$ and
the antipode:
$$ S_\phi(a) = S(a), \qquad S_\phi(b) = e^{i\phi} S(b), $$
with the only nontrivial braiding:
$$ \Psi(b,b) = e^{2i\phi} b \otimes b. $$

\subsection{The Hopf algebra  $\A(SU_q(n))$.}\ \newline
The $\A(SU_q(n))$ algebra is generated by a unitary matrix elements $u_{ij}, 1 \leq i,j \leq n,$
with relations:
$$ 
\begin{aligned}
&u_{ik} u_{jk} = q u_{jk} u_{ik}, (i<j), &
& u_{ki} u_{kj} = q u_{kj} u_{ki}, (i < j), \\
& u_{il} u_{jk} = u_{jk} u_{il}, (i <j; k < l), &
& u_{ik} u_{jl} - u_{jl} u_{ik} = (q - q^{-1}) u_{jk} u_{il}, (i < j; k < l), \\
& \sum_\sigma (-q)^{|\sigma|} u_{1\sigma(1)} \cdots u_{n\sigma(n)} =1,&
\end{aligned}
$$
where $|\sigma|$ is a number of inversions in permutation $\sigma\in S_n$.

The Hopf algebra structure on $\mathcal{A}(SU_q(n))$ is given by 
\begin{center}
$\Delta(u_{ij})=\sum\limits_{k}u_{ik}\otimes u_{kj}, \ \ \ \ \ \epsilon(u_{ij})=\delta_{ij}, \ \ \ \ \ S(u_{ij})=(u_{ji})^{\ast},$
\end{center}
where the $\ast$-structure is given by
\begin{equation}
(u_{ij})^{\ast}=(-q)^{j-i}\sum\limits_{\sigma}(-q)^{|\sigma|}u_{k_1\sigma(l_1)}...u_{k_{n-1}\sigma(l_{n-1})},
\end{equation}
where
\begin{center}
$(k_1,...,k_{n-1})=(1,...,n)\setminus\{i\}, \ \ \ \ (l_1,...,l_{n-1})=(1,...,n)\setminus\{j\}$
\end{center}
treated as ordered sets.

Let $p: \{1,2,\ldots,n\} \to {\mathbb Z}$ be a function such that $\sum\limits_{k=1}^{n} p(k) = 0$.

\begin{lem}The following map:
$$ \a_p: u_{ij} \mapsto z^{p(i)} \delta_{ij}, $$
is a $\ast$-algebra homomorphism from $\A(SU_q(n))$ to $\C[z,z^{-1}]$.
\end{lem}

\begin{lem}
$\a_p$ is a coalgebra morphism.
\end{lem}
\begin{proof}
We calculate
$$(\a_p\otimes \a_p)\Delta(u_{ij})=(\a_p\otimes\a_p)\left(\sum\limits_{k} u_{ik}\otimes u_{kj}\right)=
\sum\limits_{k} z^{p(i)} \delta_{ik}\otimes z^{p(j)}\delta_{jk}=z^{p(i)}\otimes z^{p(j)}\delta_{ij}.$$
On the other hand we have 
$$\Delta_{\mathbb{C}[z,z^{-1}]}\left(\a_p(u_{ij})\right)=\Delta_{\mathbb{C}[z,z^{-1}]}\left(z^{p(i)}\delta_{ij}\right)=z^{p(i)}\otimes z^{p(j)}\delta_{ij}.$$

Moreover, 
$$(\a_p \otimes \a_p)\Delta(1)=1\otimes 1=\Delta_{\mathbb{C}[z,z^{-1}]}\left(\a_p(1)\right).$$
\end{proof}

\begin{lem}
$\a_p$ is a Hopf algebra morphism.
\end{lem}
\begin{proof}
Simple calculation, using definition of $\ast$-structure, gives us $\a_p(u_{ij}^{\ast})=z^{-p(i)}\delta_{ij}$. 
Hence 
$$(\a_p\circ S)(u_{ij})=\a_p(u_{ji}^{\ast})=z^{-p(j)}\delta_{ij}$$
but on the other hand we have 
$$\left(S_{\mathbb{C}[z,z^{-1}]}\circ \a_p \right)(u_{ij})=S_{\mathbb{C}[z,z^{-1}]}\left(z^{p(i)}\delta_{ij}\right)=z^{-p(i)}\delta_{ij}.$$
\end{proof}
Let us now calculate degrees of generators.
$$\Delta_L(u_{ij})=(\a_p\otimes \mathrm{id})\left(\sum\limits_{k}u_{ik}\otimes u_{kj}\right)=\sum\limits_{k}\delta_{ik}z^{p(i)}\otimes u_{kj}=z^{p(i)}\otimes u_{ij}.$$
Similarly $\Delta_R(u_{ij})=u_{ij}\otimes z^{p(j)}$.

Hence $\mu(u_{ij})=p(i),\ \nu(u_{ij})=p(j)$ and
$$\delta(u_{ij})=p(i)-p(j).$$
Therefore, by the general construction we obtain
\begin{cor}
The algebra $\A_\phi(SU_q(n))$ is a braided Hopf algebra with the product:
$$ u_{ij}\ast u_{kl} = e^{i\phi\left(p(i)p(l)-p(k)p(j)\right)} u_{ij}\cdot u_{kl}, $$
the coproduct:
$$ \Delta_\phi(u_{ij})=\sum\limits_{k}e^{i\phi\left(p(i)-p(k)\right)\left(p(k)-p(j)\right)}
u_{ik}\otimes u_{kj}, $$
the braiding:
$$ \Psi(u_{ij}\otimes u_{kl})=e^{2 i\phi\left(p(i)-p(j)\right) \left(p(k)-p(l)\right)}
u_{kl}\otimes u_{ij},$$
the same counit and with the antipode:
$$ S_\phi(u_{ij})=e^{i\phi\left(p(i)-p(j)\right)^2}S(u_{ij}). $$
\end{cor}
\subsection{The $\A(SU_q(2))$ Hopf algebra.}
\ \newline
Here we discuss in details the case of the braided Hopf algebra $\mathcal{A}(SU_q(2))$
and its explicit presentation. The results could be directly compared with the
presentation in \cite{KMRW14}.

Let $q$ be a real number with $0 < q < 1$, and let $\A = \A(SU_q(2))$ be the $*$-algebra 
generated by $a$ and $b$, subject to the following commutation rules:
\begin{align}
& a b = q b a,  \qquad  a b^*= q b^* a, \qquad bb^* = b^*b,  \\
& a^*a +  b^*b = 1,  \qquad  aa^* + q^2 b^*b = 1.
\label{eq:suq2-relns}
\end{align}
As a consequence, $a^*b = q^{-1} ba^*$ and $a^*b^* = q^{-1} b^* a^*$. 
This becomes a Hopf $*$-algebra under the coproduct
\begin{align*}
\cop a &:= a \ox a - q\,b^* \ox b,
\\
\cop b &:= b \ox a + a^* \ox b, 
\end{align*}
counit $\epsilon(a) = 1$, $\epsilon(b) = 0$, and the antipode 
$$ Sa = a^*,\quad Sb = - qb, \quad Sb^* = - q^{-1}b^*, \quad Sa^* = a.$$

Let $\a: \A \to C(S^1)$ be the following map:
$$ \a(a) = z, \quad \quad \a(b) = 0. $$
It is an Hopf algebra homomorphism from $\a$ to $\C[z,z^{-1}]$, so we have:

\begin{lem}
The following establishes the grading of homogeneous elements as defined in the previous
section:
$$ \mu(a) = 1 = \nu(a), \qquad \qquad \mu(b) = -1 = -\nu(b). $$ 
\end{lem}

\begin{thm}
Let $0 \leq \phi < 2 \pi$. The deformed $\ast$-algebra $\A_\phi(SU_q(2))$
is isomorphic to the algebra with the following generators and relations:
\begin{align}
&\alpha^* \alpha +\gamma^* \gamma= 1, \qquad
\alpha\alpha^*+ q^2 \gamma^* \gamma =1, \\
&\gamma\gamma^*=\gamma^*\gamma \\
&\alpha \gamma = q e^{4 i\phi}\gamma  \alpha , \qquad
\alpha \gamma^*  = q e^{-4 i \phi} \gamma^* \alpha .
\end{align}
\end{thm}
\begin{proof}
We compute:
$$
\begin{aligned}
& a \ast a^* = a a^*, \qquad a^* \ast a = a^* a, \\
& b \ast b^* = bb^*, \qquad  b^* \ast b = b^*b, \\
& a \ast b =  e^{2 i \phi} ab, \qquad 
  a \ast b^* = e^{-2 i \phi} b^* \ast a, \\
& b \ast a =  e^{-2 i \phi} ba, \qquad 
  b^* \ast a = e^{2 i \phi} a \ast b^*.
\end{aligned}
$$
If $\rho$ is the identity map $\rho: H \to H_\phi$, setting $\rho(a) = \alpha$ and 
$\rho(b) = \gamma$ we obtain from the relations \ref{eq:suq2-relns} the relations 
of $\A_\phi(SU_q(2))$.
\end{proof}
For $\A(SU_q(2))$ we have:
$$ \delta(a) = 0 = \delta(a^*), \qquad \delta(b^*) = 2 = - \delta(b).$$
Next, we can see how the coproduct is changed for $\A_\phi(SU_q(2))$:
\begin{align}
&\Delta_\phi(\alpha)=\alpha \otimes \alpha - 
 q e^{-4i\phi} \gamma^*\otimes \gamma\\
&\Delta_\phi(\alpha^*)=\alpha^*\otimes \alpha^* - 
 q e^{-4i \phi} \gamma \otimes \gamma^*\\
&\Delta_\phi(\gamma)=\gamma\otimes\alpha +\alpha^*\otimes \gamma \\
&\Delta_\phi(\gamma^*)=\gamma^*\otimes \alpha^*+\alpha\otimes \gamma^*
\end{align}

In the $\A(SU_q(2))$ case the only nontrivial braiding phase factors $\psi$ between the generators $\alpha, \gamma$ are:

$$ \psi(\gamma, \gamma) = e^{8 i \phi} = \psi(\gamma^*, \gamma^*), \qquad
   \psi(\gamma^*, \gamma) = e^{-8 i \phi} = \psi(\gamma, \gamma^*).
$$ 
We compute the antipode on the the generators of $\A(SU_q(2)$: 
\begin{align}
&S_{\phi}(\alpha)=\alpha^*, \qquad &S_{\phi}(\alpha^*)=\alpha \\
&S_{\phi}(\gamma)= - q e^{ 4i\phi} \gamma, &S_{\phi}(\gamma^*)= - q^{-1} e^{ 4i\phi}  \gamma^*. 
\end{align}

\subsection{The quantum double torus.}
\ \newline

Though the noncommutative torus is not a Hopf algebra, there exists a Hopf algebra structure over 
a direct sum of the commutative and noncommutative torus described by Hajac and Masuda 
\cite{HM98}.

Let $\mathcal{A}=C(\mathbb{T}^2)\oplus C(\mathbb{T}_q^2)$ with generators $u,v$ 
and $U,V$ of $C(\mathbb{T}^2)$ and $C(\mathbb{T}^2_q)$, respectively. 
With the following coproduct:
$$
\begin{aligned}
\Delta(u)=u\otimes u+V\otimes U, \qquad & 
\Delta(v)=v\otimes v+U\otimes V, \\
\Delta(U)=U\otimes u +v\otimes U, \qquad &
\Delta(V)=V\otimes v +u\otimes V,
\end{aligned}
$$
counit
$$
\epsilon(u)=\epsilon(v)=1, \qquad 
\epsilon(U)=\epsilon(V)=0
$$
and the antipode:
$$ 
S(u)=u^{\ast}, \quad S(v)=v^*,\quad S(U)=V^{\ast}, \quad S(V)=U^{\ast},
$$ 
$\mathcal{A}$ is a Hopf algebra called a quantum double torus.

\begin{lem}
Let us define $\a:\mathcal{A}\rightarrow \mathbb{C}[z,z^{-1}]$ by 
$$ \a(u)=z,\quad \a(v)=z^{-1}, \quad \a(U)=\a(V)=0.$$ 
Then $\a$ is a Hopf algebra map.

Using the general construction we obtain a braided Hopf algebra 
$\mathcal{A}_{\phi}$ with the following properties:
\begin{enumerate}
\item as algebras $\mathcal{A}\cong \mathcal{A}_\phi$,
\item the coproduct on $\mathcal{A}_\phi$ is as follows:
\begin{align}
&\Delta_\phi(x) = \Delta(x), \quad x \in C(\mathbb{T}_q^2), \\
&\Delta_\phi(u) = u\otimes u +e^{-4i\phi}V\otimes U, \\
&\Delta_{\phi}(v)=v\otimes v +e^{-4i\phi}U\otimes V, 
\end{align}
\item the only non-trivial braiding is on the noncommutative torus alone,
using the shorthand $\Psi(x \ox y) = \psi(x,y) y \ox x$:
\begin{align*}
& \psi(U,U)=\psi(V,V)=\psi(U^*,U^*)=\psi(V^*,V^*)=e^{8i\phi}, \\
& \psi(U,V)=\psi(V,U)=\psi(U,U^*)=\psi(U^*,U)=\psi(V,V^*)=\psi(V^*,V)= e^{-8i\phi}, \\
& \psi(U,V^*)=\psi(V^*,U)=\psi(U^*,V)=\psi(V,U^*)= e^{8i\phi}. 
\end{align*}
\end{enumerate}
\end{lem}
So we can see that the same algebra has a family of coproducts
each in a different braided category, which make it a braided Hopf algebra.
\section{Bicovariant differential calculi}

Recall that a left action of a (braided) Hopf algebra $H$ on vector space $\Gamma$ is a linear map $\rhd:H\otimes \Gamma\rightarrow \Gamma$ s.th. $(hg)\rhd x=h\rhd(g\rhd x)$ and $1\rhd x=x$ for all $g,h\in H, x\in \Gamma$. Similarly we define right action. Then we say that $\Gamma$ is a left (resp. right) $H$-module. 

 A left coaction of $H$ on $\Gamma$ is a linear map $\delta_L: \Gamma\rightarrow H\otimes \Gamma$ satisfying

$$ \left(\Delta\otimes \mathrm{id}\right)\circ 
\delta_L= \left(\mathrm{id}\otimes \delta_L\right)\circ\delta_L,$$ 
and 

$$\left(\epsilon\otimes \mathrm{id}\right)\delta_L=\mathrm{id}.$$

Analogously we define the right one. We shall say $\Gamma$ is a left (resp. right) comodule and very often we shall use Sweedler's notation:
$$ \delta_L (x) = x_{(-1)} \otimes x_{(0)}, \qquad  \delta_R (x) = x_{(0)} \otimes x_{(1)}. $$
 
If $\Gamma$ is both left and right $H$-module, with the actions that are commutative with each other:
$$ (h \rhd x) \lhd g = h \rhd (x \lhd g),$$
then we call $H$-bimodule. We define $H$-bicomodule in  a similar way.
  
\begin{defn}
We say that $\Gamma$ is a left covariant bimodule if it is a left $H$-bimodule which is a left comodule over $H$ with coaction $\delta_L$ such that for any $a,b\in H, \ \rho\in \Gamma$ the following are satisfied
$$
\begin{aligned}
&&\delta_L(a\,\rho)=\Delta(a) \delta_L(\rho), \qquad
&\delta_L(\rho \, a)=\delta_L(\rho)\Delta(a).
\end{aligned}
$$
\end{defn}
In the above relations we used shorter notation:
$x \lhd \rho = x\, \rho$ and $\rho \rhd x = \rho \, x$. The multiplication $(x\ox y)(z\ox \rho)$ for $x,y,z\in H,\ \rho\in \Gamma$ is defined by:
$$(x \ox y)(z\ox \rho)=(m\ox \rhd)(x\ox \tilde\Psi(y\ox z)\ox \rho),$$
where $\tilde{\Psi}$ is the braiding of the braided Hopf algebra $H$. 
Similarly for the right module structure:
$$(z \ox \rho) (x \ox y) = (m \ox \lhd)(z \ox \hat\Psi(\rho \ox x)\ox y),$$
where $\hat{\Psi}: \Gamma \otimes H \to H \otimes \Gamma$ is 
braiding between $H$ and $\Gamma$. 

Similarly we define right covariant bimodule with compatibility conditions
$$
\begin{aligned}
&&\delta_R(a\,\rho)=\Delta(a) \delta_R(\rho), \qquad
&\delta_R(\rho \, a)=\delta_R(\rho)\Delta(a).
\end{aligned}
$$

Note, that the above conditions require that $\Gamma \ox H$ and $H \ox \Gamma$
have a bimodule structure over $H \ox H$. This is obvious in the case $H$ is not
braided, however, if $H$ is a braided Hopf algebra it requires an introduction of the braiding between $\Gamma$ and $H$.

\begin{defn}
We say that $\Gamma$ is a bicovariant bimodule or Hopf bimodule over $H$ if it is both $H$-bicomodule, left and right covariant bimodule.
\end{defn}

Let us recall the definition of left and right covariant differential calculus for
a braided Hopf algebra. Let $H$ be a braided Hopf algebra and $(\Gamma, d)$ be the
first order differential calculus (FODC), that is $\Gamma$ is a Hopf bimodule over 
$H$ and $d: H \to \Gamma$ satisfies the Leibniz rule and is surjective in the sense 
that every element of $\Gamma$ is of the form $\sum a_k\rhd db_k$ for some 
$a_k,b_k\in H$.

\begin{defn} We say that the first order differential calculus $(\Gamma, d)$ 
	over a braided Hopf
	algebra $H$ is  {\em braided left covariant} if $\Gamma$ is a Hopf bimodule over $H$ and
	$$ \delta_L (d h) = (\id \otimes d) \Delta(h), \;\;\; \forall h \in H. $$
\end{defn}

In a similar way we define the braided right covariant calculus. We say that the
calculus is braided bicovariant if it is left and right covariant.

\subsection{The universal differential calculus}

Let $H$ be a braided Hopf algebra with the braiding $\Psi$ 
and $\Gamma_u = \hbox{ker}\, m \subset H \otimes H$ 
be the Hopf bimodule of the universal first order differential 
calculus, i.e. with $da=a\ox 1-1\ox a$.

\begin{lem}
The following maps,
$$
\begin{aligned}
&&\delta_L: \Gamma_u \rightarrow H \otimes \Gamma_u, \qquad
&\delta_L(x \otimes y) = (m \otimes \id \otimes \id)  (x_{(1)} \otimes 
\Psi( x_{(2)} \otimes  y_{(1)}) \otimes y_{(2)}), \\
&&\delta_R: \Gamma_u \rightarrow \Gamma_u \otimes H, \qquad
&\delta_R(x \otimes y) =  (\id \otimes \id \otimes\, m) (x_{(1)} \otimes \Psi( x_{(2)} \otimes  y_{(1)}) \otimes y_{(2)}),
\end{aligned}
$$
make $(\Gamma, D)$ a bicovariant FODC.
\end{lem}

\begin{proof}
Using fact that if on one leg we have unity then $\Psi$ acts on such tensor product as usual flip, we obtain the following
$$
\begin{aligned}
&& \delta_L(a\ox 1)=a_{(1)}\ox a_{(2)}\ox 1 , \qquad & \delta_L(1\ox a)=a_{(1)}\ox 1\ox a_{(2)}, \\
&& \delta_R(a\otimes 1)=a_{(1)}\ox 1 \ox a_{(2)}, \qquad & \delta_R(1\ox a)=1\ox a_{(1)}\ox a_{(2)}.
\end{aligned}
$$
Hence we obtain 
$$
\begin{aligned}
&& \delta_L(da)=a_{(1)}\ox\left(a_{(2)}\otimes 1-1\otimes a_{(2)}\right)=(\id\ox d)\Delta(a),\\
&& \delta_R(da)=\left(a_{(1)}\ox 1 -1\ox a_{(1)}\right)\ox a_{(2)}=(d\ox\id)\Delta(a).
\end{aligned}
$$
\end{proof}

\subsection{Braided bicovariant calculi over twisted braided Hopf algebras.}

Let us assume now that we have a braided Hopf algebra $H_\phi$ constructed from 
a Hopf algebra $H$ (as stated in Theorem \ref{braidedHopf}) and let $(\Gamma, d)$ be 
a bicovariant first order differential calculus over $H$. 

Then, $\Gamma$ has a left and right coaction of $\C[z,z^{-1}]$ defined 
in a similar way as on $H$:

$$ \hat{\delta}_L( \omega) = (\a \ox \id) \delta_L (\omega), \qquad
\hat{\delta}_R( \omega) = (\id \ox \a) \delta_R (\omega).$$

We say that an element $\omega$ is homogeneous of degrees 
$\hat{\mu},\hat{\nu}$ iff:
$$ \hat{\delta}_L(\omega) = z^{\hat{\mu}(\omega)} \ox \omega, \qquad
 \hat{\delta}_R(\omega) = \omega \ox z^{\hat{\nu}(\omega)}.$$

We have:
$$ \hat{\mu}(dh) = \mu(h), \qquad \hat{\nu}(dh) = \nu(h), $$
$$ \hat{\mu}(x\, \omega) = \mu(x)+\hat{\mu}(\omega), \qquad \hat{\nu}(x\, \omega) = \nu(x) +\hat{\nu}(\omega), $$
$$ \hat{\mu}(\omega\, y) = \hat{\mu}(\omega)+{\mu}(y), \qquad \hat{\nu}(\omega\, y) = \hat{\nu}(x) +{\nu}(\omega). $$
Indeed, using fact that $\a $ is Hopf algebra morphisms we have
$$
\begin{aligned}
\left(\Delta\ox \id\right)\hat{\delta}_L &=\left(\Delta\circ\a \ox\id\right)\delta_L= \left(\a \ox\a \ox\id\right)\left(\Delta\ox\id\right)\delta_L= \left(\a \ox\a \ox\id\right)\left(\id\ox \delta_L\right)\delta_L= \\
&=\left(\id\ox\left(\a \ox\id\right)\delta_L\right)\left(\a \ox\id\right) \delta_L=\left(\id\ox\hat{\delta}_L\right)\hat{\delta}_L
\end{aligned}
$$
Similarly for the right one. \\
Next, suppose that $h$ has degrees $\mu(h),\nu(h)$, i.e.
$$\begin{aligned}
 \left(\a \ox\id\right)\Delta(h)=z^{\mu(h)}\ox h \qquad \mathrm{and} \qquad \left(\id\ox\a \right)\Delta(h)=h\ox z^{\nu(h)}.
 \end{aligned}$$
 Then 
 $$\begin{aligned}
 \hat{\delta}_L(dh)&=\left(\a \ox\id\right)\delta_L(dh)= \left(\a \ox\id\right)\left(\id\ox d\right)\Delta(h)= \left(\id\ox d\right)\left(\a \ox\id\right)\Delta(h)= \\ &=\left(\id\ox d\right)\left(z^{\mu(h)}\ox h \right)=z^{\mu(h)}\ox h.\end{aligned}$$
 Hence $\hat{\mu}(dh)=\mu(h)$ and similarly for the right degree, $\hat{\nu}(dh) = \nu(h)$. \\
 Now, left covariance of the Hopf bimodule $\Gamma$ means 
 $\delta_L(x\omega)=\Delta(x)\delta_L(\omega)$. It implies the following
 $$
 \begin{aligned}
  \hat{\delta}_L(x\omega)&=\left(m\ox\rhd\right)\left(\a \ox\tau\ox \id\right)\left(\Delta\ox\left(\a  \ox\id\right)\delta_L\right)\left( x\ox \omega\right)= \\ &=\left(m\ox\rhd\right)\left(\id\ox\tau\ox \id\right) \left(\left(\a \ox\id\right)\Delta \ox (\a \ox\id)\delta_L\right)\left(x\ox\omega\right)=
  \\ &=z^{\mu(x)+\hat{\mu}(\omega)}\ox x\omega. 
 \end{aligned}
 $$
 Analogously we can prove other equalities.

\begin{lem}\label{bimodb}
Taking  $\Gamma_\phi=\Gamma$ as a vector space, the following defines 
a left and right module structure on $\Gamma_\phi$ over $H_\phi$:
$$ 
\begin{aligned}
x \ast \omega = e^{i\phi  ( \mu(x) \hat{\nu}(\omega) - \nu(x) \hat{\mu}(\omega))} \, x \, \omega, \\
\omega \ast x  = e^{i\phi  ( \hat{\mu}(\omega)) \nu(x) - \hat{\nu}(\omega) \mu(x) } \, \omega \, x.
\end{aligned}
$$
\end{lem}
\begin{proof}
Using the properties of $\hat{\mu},\mu$ and $\hat{\nu},\nu$, we obtain
analogous cocycle conditions as for $\mu,\nu$ alone in Lemma \ref{cocycle}:
$$
\begin{aligned}
& \mu(x)\hat{\nu}(y\omega)-\nu(x)\hat{\mu}(y,\omega)+\mu(y)\hat{\nu}(\omega)-\nu(y)\hat{\mu}(\omega)=\\
&=\mu(xy)\hat{\nu}(\omega)-\nu(xy)\hat{\mu}(\omega)+\nu(x)\hat{\nu}(y)-\nu(x)\hat{\mu}(y), \\
& \hat{\mu}(\omega)\nu(xy)-\hat{\nu}(\omega)\mu(xy)+\mu(x)\nu(y)-\nu(x)\mu(y) =\\
& =\hat{\mu}(\omega x)\nu(y)-\hat{\nu}(\omega x)\mu(y)+\hat{\mu}(\omega)\nu(x)-\hat{\nu}(\omega)\mu(x).
\end{aligned}
$$
That implies that we obtain left and right module structure.
\end{proof}
\begin{lem}\label{bibim}
The following defines the left (right) coaction of $H_\phi$ on $\Gamma_\phi$ on 
homogeneous elements as:
$$ \delta_{L^\phi}(\omega)=  e^{i\phi {\delta}(\omega_{(-1)})\hat{\delta}(\omega_{(0)})}  
\omega_{(-1)} \ox \omega_{(0)},$$ 
and
$$ \delta_{R^\phi}(\omega)=  e^{i\phi  \hat{\delta}(\omega_{(0)})\delta(\omega_{(1)})}  
\omega_{(0)} \ox \omega_{(1)}.$$ 

with the braiding between the bimodule $\Gamma_\phi$ and $H_\phi$ defined in the
same way:
$$ 
\begin{aligned}
&\Psi(\omega \ox x) = e^{2 i\phi  \hat{\delta}(\omega)\delta(x)}  x \ox \omega, \\
&\Psi(x \ox \omega) = e^{2 i\phi  \hat{\delta}(\omega)\delta(x)}  \omega \ox x 
\end{aligned}
$$
\end{lem}

\begin{proof}
The proof follows exactly the one for the the braided Hopf algebra $H_\phi$.
\end{proof}

\begin{thm}
With the above definitions $(\Gamma_\phi, d)$ is a braided bicovariant differential calculus.
\end{thm}
\begin{proof}
We have already demonstrated that $\Gamma_\phi$ is a Hopf bimodule, so it 
remains to show that the external derivative $d$ is compatible with the covariance. This
is, however, an immediate consequence of Lemma \ref{bibim}. Indeed, if we have
a homogeneous $x$ then:
$$ (\hbox{id} \otimes d) \Delta_\phi(x) = e^{i\phi \delta(x_{(1)}) \delta(x_{(2)})} x_{(1)} \otimes dx_{(2)}.$$
On the other hand using Lemma \ref{bibim}:
$$ \delta_{L^\phi} (dx) = e^{i\phi \delta(dx_{(-1)}) \hat{\delta}((dx)_{(0)}} (dx)_{(-1)} \otimes
	(dx)_{(0)},$$
and combining it with Lemma \ref{bimodb} and the fact that $(\Gamma,d)$ is a
bicovariant first order differential calculus we obtain that the same holds for the
braided $(\Gamma_\phi,d)$. 
\end{proof}
\section{Outlook}
In this note we have presented a construction of class of braided Hopf algebras, which
are obtained by twisting. It is interesting to observe that this includes the braided
$SU_q(2)$ algebra and therefore, in particular, a family of so-called theta deformations
of the sphere $S^3$, introduced by Matsumoto \cite{Ma91}. These objects have been so far considered only as
algebras corresponding to noncommutative homogenoeus spaces on which some (unbraided) Hopf 
algebras act and coact. 

The existence of braided Hopf algebra symmetries is a new element that could possibly have some 
implications for the geometries of theta-deformed three spheres. In particular, it
would be interesting to verify that the Connes' spectral geometries and spectral triples
for the theta-spheres \cite{CoLa} are compatible with the structure of the braided Hopf algebras.

{\bf Acknowledgements:} AB acknowledges support from a grant from the John Templeton Foundation.

 
\end{document}